\documentclass{amsart}
\usepackage{amsfonts,amsmath,amsthm,amssymb,}
\usepackage[latin1]{inputenc}
\usepackage{graphicx}
\usepackage{url}

\usepackage{fancyhdr}

\usepackage{enumerate}
\usepackage[breaklinks]{hyperref}
\usepackage{verbatim}
\usepackage[mathscr]{euscript}
\usepackage{enumerate,multicol,array}

\title[Forking and stability in a $C^*$-algebra representation]{Forking and stability in the representations of a $C^*$-algebra}

\author{Camilo Argoty}
\address{Camilo Argoty
\\ Departamento de Matem\'aticas
\\ Universidad Sergio Arboleda
	\\ Bogot\'a, Colombia}

\date{}

\thanks{The author is very thankful to Alexander Berenstein for his help in reading and correcting this work. In the same way, the author wants to thank Ita\"i Ben Yaacov for discussing ideas in particular cases. Finally, the author wants to thank Andr\'es Villaveces for his interest and for sharing ideas on how to extend the results of this paper.}

\makeatletter
\def\newrefformat#1#2{%
  \@namedef{pr@#1}##1{#2}}
\def\prettyref#1{\@prettyref#1:}
\def\@prettyref#1:#2:{%
  \expandafter\ifx\csname pr@#1\endcsname\relax%
    \PackageWarning{prettyref}{Reference format #1\space undefined}%
    \ref{#1:#2}%
  \else%
    \csname pr@#1\endcsname{#1:#2}%
  \fi%
}
\makeatother

\newrefformat{eq}{\textup{(\ref{#1})}}
\newrefformat{cha}{Chapter \ref{#1}}
\newrefformat{sec}{Section \ref{#1}}
\newrefformat{tab}{Table \ref{#1} on page \pageref{#1}}
\newrefformat{fig}{Figure \ref{#1} on page \pageref{#1}}

\makeatletter
\def\indsym#1#2{%
  \setbox0=\hbox{$\m@th#1x$}%
  \kern\wd0%
  \hbox to 0pt{\hss$\m@th#1\mid$\hbox to 0pt{$\m@th#1^{#2}$}\hss}%
  \lower.9\ht0\hbox to 0pt{\hss$\m@th#1\smile$\hss}%
  \kern\wd0} \newcommand{\ind}[1][]{\mathop{\mathpalette\indsym{#1}}}
\def\nindsym#1#2{%
  \setbox0=\hbox{$\m@th#1x$}%
  \kern\wd0%
  \hbox to 0pt{\mathchardef\nn="3236\hss$\m@th#1\nn$\kern1.4\wd0\hss}
  \hbox to 0pt{\hss$\m@th#1\mid$\hbox to 0pt{$\m@th#1^{#2}$}\hss}%
  \lower.9\ht0\hbox to 0pt{\hss$\m@th#1\smile$\hss}%
  \kern\wd0}
\newcommand{\nind}[1][]{\mathop{\mathpalette\nindsym{#1}}}
\makeatother

 \def\bv{\bar v}  
   
 \def\bw{\bar w}
\def\p{\pi}\def\f{\phi}

\def\ben{\begin{enumerate}}\def\een{\end{enumerate}}
\def\bdc{\begin{description}}\def\edc{\end{description}}
\def\bitm{\begin{itemize}}\def\eitm{\end{itemize}}
\def\bdf{\begin{defin}}\def\edf{\end{defin}}
\def\bth{\begin{theo}}\def\eth{\end{theo}}
\def\bfc{\begin{fact}}\def\efc{\end{fact}}
\def\bco{\begin{coro}}\def\eco{\end{coro}}
\def\brm{\begin{rem}}\def\erm{\end{rem}}
\def\blm{\begin{lemma}}\def\elm{\end{lemma}}
\def\bnt{\begin{nota}}\def\ent{\end{nota}}
\def\bex{\begin{exe}}\def\eex{\end{exe}}
\def\bpf{\begin{proof}}\def\epf{\end{proof}}
\def\bas{\begin{assum}}\def\eas{\end{assum}}
\def\beq{\begin{equation}}\def\eeq{\end{equation}}
\def\becl{\begin{cla}}\def\ecl{\end{cla}}

 \def\r{\rho}
 \def\k{\kappa}  
  \def\g{\gamma} \def\e{\epsilon} \def\l{\lambda}  
   \def\f{\phi}

\def\A{\mathcal{A}}

\def\O+{\oplus}

\newtheorem{theo}{Theorem}[section]

\newtheorem{coro}[theo]{Corollary}
\newtheorem{lemma}[theo]{Lemma}

\theoremstyle{definition}
\newtheorem{defin}[theo]{Definition}
\newtheorem{fact}[theo]{Fact}
\theoremstyle{remark}
\newtheorem{exe}[theo]{Example}
\newtheorem{rem}[theo]{Remark}
\newtheorem{assum}[theo]{Assumption}
\newtheorem{nota}[theo]{Notation}
\newtheorem*{cla}{Claim}

\newcommand{\N}{\ensuremath{\mathbb{N}}}

\newcommand{\C}{\ensuremath{\mathbb{C}}}

\newcommand{\MM}{\ensuremath{\mathcal{M}}}
\newcommand{\Aut}{\ensuremath{\text{Aut}}}

\newcommand{\sfrac}[2]{\hbox{$\frac{#1}{#2}$}}
\newcommand{\half}[1][1]{\sfrac{#1}{2}}

\begin{document}

\begin{abstract}
We show that the theory of a non-degenerate representation of a $C^*$-algebra $\A$ over a Hilbert space $H$ is superstable. Also, we characterize forking, orthogonality and domination of types and show that the theory has weak elimination of imaginaries.
\end{abstract}

\maketitle

\section{introduction}

Let $\A$ be a unital $C^*$-algebra and let $\p:\A\to B(H)$ be a $C^*$-algebra nondegenerate isometric homomorphism, where $B(H)$ is the algebra of bounded operators over a Hilbert space $H$. We continue our study of $H$ as a metric structure expanded by $\A$ from the point of view of continuous logic. The basic model theoretic facts of the structure were proved in \cite{Ar} and in this paper we deal with stability theoretic properties. Recall that we include a symbol $\dot{a}$ in the language of the Hilbert space structure whose interpretation in $H$ will be $\p(a)$ for every $a$ in the unit ball of $\A$ getting the following metric structure of only one sort: 
\[(Ball_1(H),0,-,i,\half[x+y],\|\cdot\|,(\p(a))_{a\in Ball_1(\A)})\] where $Ball_1(H)$  and $Ball_1(\A)$ are the corresponding unit balls in $H$ and $\A$ respectively; $0$ is the zero vector in $H$; $-: Ball_1(H)\to Ball_1(H)$ is the function that to any vector $v\in Ball_1(H)$ assigns the vector $-v$; $i: Ball_1(H)\to Ball_1(H)$ is the function that to any vector $v\in Ball_1(H)$ assigns the vector $iv$ where $i^2=-1$; $\half[x+y]: Ball_1(H)\times Ball_1(H)\to Ball_1(H)$ is the  function that to a couple of vectors $v$, $w\in Ball_1(H)$ assigns the vector $\half[v+w]$; $\|\cdot\|: Ball_1(H) \to [0, 1]$ is the norm function; $\A$ is an unital $C^*$-algebra; $\p:\A\to B(H)$ is a $C^*$-algebra isometric homomorphism. The metric is given by $d(v,w) =\|\half[v-w]\|$. Briefly, the structure will be refered to as $(H,\p)$. 

As we said before, this is a continuation of a previous work (see \cite{Ar}). The main results obtained there are the following:
\ben
\item The theory of $(H,\p)$ has quantifier elimination. Furthermore, for $v,w\in H$, $tp(v/\emptyset)=tp(w/\emptyset)$ if and only if $\f_v=\f_w$. Here, $\f_v:\A\to\C$ is the positive linear functional defined by the formula $\f_v(a)=\langle av\ |\ v\rangle$.
\item An explicit description of the model companion of $Th(H,\p)$.
\een

In this paper we get the following results

\ben
\item Characterize non-forking in $Th(H,\p)$.
\item $Th(H,\p)$ is superstable.
\item Let $\bv\in H^n$ and $E\subseteq H$. Then the type $tp(\bv/E)$ has a canonical base formed by a tuple of elements in $H$ and therefore, the theory of $(H,\p)$ has weak elimination of imaginaries.
\item Let $E\subseteq H$, $p,q\in S_1(E)$ be stationary and $v$, $w\in H$ be such that $v\models p$ and $w\models q$. Then, $p\perp_E q$ if and only if $\f_{P^\perp_{acl(E)}(v_e)}\perp \f_{P^\perp_{acl(E)}(w_e)}$, where the orthogonality of functionals is defined in Definition \ref{relationsbetweenpositivelinearfunctionals}.
\een

Let us recall previous work around this topic. In \cite{Io}, Henson and Iovino,  observed that the theory of a Hilbert space expanded with a family of bounded operators is stable. A geometric characterization of forking in Hilbert spaces expanded with normal commuting operators, was first done by Berenstein and Buechler \cite{BeBue}. In \cite{BUZ} Ben Yaacov, Usvyatsov and Zadka characterized the unitary operators corresponding to generic automorphisms of a Hilbert space as those unitary transformations whose spectrum is $S^1$ and gave the key ideas used in this paper to characterize domination and orthogonality of types. The author and Ben Yaacov (\cite{ArBen}) studied the more general case of a Hilbert space expanded by a normal operator $N$. The author and Berenstein (\cite{ArBer}) studied the theory of the structure $(H,+,0,\langle|\rangle,U)$ where $U$ is a unitary operator in the case where the spectrum is countable and characterized prime models and orthogonality of types.  Most results in this paper are generalizations of results present in \cite{ArBer,ArBen}.   


This paper is divided as follows: In section \ref{SectionDefinableAndAlgebraicClosures}, we characterize definable and algebraic closures. In Section \ref{Forking}, we give a geometric interpretation of forking and show weak elimination of imaginaries. Finally, in Section \ref{OrthogonalityAndDominationOfTypes}, we characterize orthogonality and domination of types.

\section{definable and algebraic closures}\label{SectionDefinableAndAlgebraicClosures}
In this section we give a characterization of definable and algebraic closures. Gelfand-Naimark-Segal construction is a tool for understanding definable closures (see Theorem 3.11 in \cite{Ar}). 

Recall the following conventions:
\begin{nota}\label{DefinitionHv}
Given $E\subseteq H$ and $v\in H$, we denote by:
\bitm 
\item $H_E$, the Hilbert subspace of $H$ generated by the elements $\p(a)v$, where $v\in E$ and $a\in\A$.
\item $\p_E:=\{\p(a)\upharpoonright H_E\ |\ a\in \A\}$.
\item $(H_E,\p_E)$, the subrepresentation of $(H,\p)$ generated by $E$.
\item $H_v$, the space $H_E$ when $E=\{v\}$ for some vector $v\in H$.
\item $\p_v:=\p_E$ when $E=\{v\}$.
\item $(H_v,\p_v)$, the subrepresentation of $(H,\p)$ generated by $v$.
\item $(H,\p,v)$, means that $\{\p(a)v\ |\ a\in\A\}$ is dense in $(H,\p)$.
\item $H_E^\perp$, the orthogonal complement of $H_E$.
\item $P_E$, the projection over $H_E$.
\item $P_{E^\perp}$, the projection over $H_E^\perp$.
\eitm
\ent

\bdf
Given a representation $\p:\A\to B(H)$, we define:
\bitm
\item The \textit{essential part} of $\p$ is the $C^*$-algebra homomorphism, \[\p_e:=\r\circ\p:\A\to B(H)/\mathcal{K}(H)\] of $\p(\A)$, where $\r$ is the canonical proyection of $B(H)$ onto the Calkin Algebra $B(H)/\mathcal{K}(H)$.
\item The \textit{discrete part} of $\p$ is the restriction,
\begin{align*}
\p_d:ker(\p_e)&\to \mathcal{K}(H)\\
a &\to \p(a)
\end{align*}
\item The \textit{discrete part} of $\p(\A)$ is defined in the following way:
\[\p(\A)_d:=\p(\A)\cap \mathcal{K}(H).\]
\item The \textit{essential part} of $\p(\A)$ is the image $\p(\A)_e$ of $\p(\A)$ in the Calkin Algebra. 
\item The \textit{essential part} of $H$ is defined in the following way:
\[H_e:=\text{ker}(\p(\A)_d)\]
\item The \textit{discrete part} of $H$ is defined in the following way:
\[H_d:=\text{ker}(\p(\A)_d)^\perp\]
\item The \textit{essential part} of a vector $v\in H$ is the projection $v_e$ of $v$ over $H_e$.
\item The \textit{discrete part} of a vector $v\in H$ is the projection $v_d$ of $v$ over $H_d$.
\item The \textit{essential part} of a set $E\subseteq H$ is the set
\[E_e:=\{v_e\ |\ v\in E\}\]
\item The \textit{discrete part} of a set $G\subseteq H$ is the set
\[E_d:=\{v_d\ |\ v\in G\}\]
\eitm
\end{defin}

\begin{fact}[Lemma 4.5 in \cite{Ar}]\label{Aut(H/E)}
Let $E\subseteq H$, $U\in \Aut(H,\p)$. Then $U\in \Aut((H,\p)/E)$ if and only if:
\bitm
\item For all $v\in H$ and all $a\in\A$, $U(\p(a)(v))=\p(a)(Uv)$.
\item For all $w\in H_E$, $Uw=w$.
\eitm
\end{fact}
	
\begin{theo}\label{definclosure}
Let $E\subseteq H$. Then $dcl(E)=H_E$
\end{theo}
\begin{proof}
From Fact \ref{Aut(H/E)}, it is clear that $H_E\subseteq dcl(E)$. On the other hand, if $v\not\in H_E$, let $\l\in \C$ such that $\l\neq 1$ and $|\l|=1$. Then, the operator $U:=Id_{H_E}\oplus \l Id_{H_E^\perp}$ is an automorphism of $(H,\p)$ fixing $E$ such that $Uv\neq v$.
\end{proof}

\bco \label{CoroElementaryEquivalent}
For any index set $I$, $(H,\p)\equiv (H,\p)\oplus(H_e,\p_e)$.
\eco
\bpf
By Fact \ref{HensonsLemma}.
\epf

\blm\label{Henotalgebraicoveremptyset}
Let $v\in H_e$. Then $v$ is not algebraic over $\emptyset$. 
\elm
\bpf
Let $\k>2^{\aleph_0}$ and consider $(H,\p)\oplus\bigoplus_{i\in I}(H_e,\p_e)$. By Corollary \ref{CoroElementaryEquivalent}. Then there are $\k$ vectors $v_i$ for $i<\k$ such that every $v_i$ has the same type over $\emptyset$ as $v$. This means that the orbit of $v$ under the automorphisms of $(H,\p)$ is unbounded and therefore $v$ is not algebraic over the emptyset.
\epf

Recall Fact 3.19 in \cite{Ar}:
\bfc[Fact 3.19 in \cite{Ar}]\label{Hdsubsetalgebraicclosureofemptyset}
Let $v\in H_d$. Then $v$ is algebraic over $\emptyset$.
\efc

\blm\label{LemmaIfVeNotZeroVNotAlgebraic}
Let $v\in H$ such that $v_e\neq 0$. Then $v$ is not algebraic over $\emptyset$.
\elm
\bpf
Clear from previous Lemma \ref{Henotalgebraicoveremptyset}.
\epf

Now, we describe the algebraic closure of $\emptyset$:
\bth\label{aclemptyset}
$acl(\emptyset)=H_d$
\eth
\bpf
By Lemma \ref{Hdsubsetalgebraicclosureofemptyset}, $H_d\subseteq acl(\emptyset)$ and, by Lemma \ref{LemmaIfVeNotZeroVNotAlgebraic}, $acl(\emptyset)\subseteq H_d$.
\epf

\bth\label{aclE}
Let $E\subseteq H$. Then $acl(E)$ is the Hilbert subspace of $H$ generated by $dcl(E)$ and $acl(\emptyset)$.
\eth
\bpf
Let $G$ be the Hilbert subspace of $H$ generated by $dcl(E)$ and $acl(\emptyset)$. It is clear that $G\subseteq acl(E)$. Let $v\in acl(E)$. By Lemma \ref{Hdsubsetalgebraicclosureofemptyset}, $v_d\in acl(\emptyset)$, and by Theorem \ref{definclosure} and Lemma \ref{Henotalgebraicoveremptyset}, $v_e\in dcl(E)\setminus acl(\emptyset)$. Then $v_e\in dcl(E)$ and $acl(E)\subseteq G$.
\epf

\section{forking and stability}\label{Forking}
In this section we give an explicit characterization of non-forking and prove that $Th(H,\p)$ is stable. Henson and Iovino in \cite{Io}, observed that a Hilbert space expanded with a family of bounded operators is stable. Here, we give an explicit description of non-forking and show that the theory is superstable.

\begin{defin}\label{AindependentB}
Let $E$, $F$, $G\subseteq H$. We say that $E$ is \textit{independent} from $G$ over $F$ if for all $v\in E$ $P_{acl(F)}(v)=P_{acl(F\cup G)}(v)$ and denote it $E\ind^*_F G$.
\end{defin}

\begin{rem}\label{independenceoverempty}\label{independenceoverA}\label{RemarkItIsEnoughIndependenceForOneVector}
Let $\bv$, $\bw\in H^n$ and $E\subseteq H$. Then it is easy to see that:
\bitm
\item $\bv$ is independent from $\bw$ over $\emptyset$ if and only if for every $j,k=1,\dots,n$, $H_{(v_j)_e}\perp H_{(w_k)_e}$.
\item $\bv$ is independent from $\bw$ over $E$ if and only if for every $j,k=1,\dots,n$, $H_{P_E^\perp(v_j)_e}\perp H_{P_E^\perp(w_k)_e}$.
\item $\bv\in H^n$ and $E$, $F\subseteq H$. Then $\bv\ind^*_EF$ if and only if for every $j=1,\dots,n$ $v_j\ind^*_EF$ that is, for all $j=1,\dots,n$ $P_{acl(E)}(v_j)=P_{acl(E\cup F)}(v_j)$.
\eitm
\erm

\begin{fact}[Theorem 4.3 in \cite{Ar}]\label{typeoverempty1}
Let $v$, $w\in\tilde{H}$. Then $tp(v/\emptyset)=tp(w/\emptyset)$ if and only if $(H_v,\p_v,v)$ is isometrically isomorphic to $(H_w,\p_w,w)$, this is, the map sending $v$ to $w$ extends to a unique isometric isomorphism between $(H_v,\p_v)$ and $(H_w,\p_w)$.
\end{fact}

\begin{theo}\label{nonforkingextension}
Let $E\subseteq F\subseteq H$, $p\in S_n(E)$ $q\in S_n(F)$ and $\bv=(v_1,\dots,v_n)$, $\bw=(v_1,\dots,v_n)\in H^n$ be such that $p=tp(\bv/E)$ and $q=tp(\bw/F)$. Then $q$ is an extension of $p$ such that $\bw\ind^*_EF$ if and only if the following conditions hold:
    \begin{enumerate}
        \item For every $j=1,\dots,n$, $P_{acl(E)}(v_j)=P_{acl(F)}(w_j)$
        \item For every $j=1,\dots,n$, $(H_{P^\perp_{acl(E)}v_j},\p_{P^\perp_{acl(E)}v_j},P^\perp_{acl(E)}v_j)$ is isometrically isomorphic to $(H_{P^\perp_{acl(F)}w_j},\p_{P^\perp_{acl(F)}w_j},P^\perp_{acl(F)}w_j)$
    \end{enumerate}
\end{theo}
\begin{proof}
Clear from Theorem \ref{typeoverempty1} and Remark \ref{independenceoverempty}
\end{proof}

\brm
Recall that for every $E\subseteq H$ and $v\in H$, $P^\perp_{acl(E)}v=(P^\perp_Ev)_e$.
\erm

\begin{theo}\label{explicitnonforkingrelation}
$\ind^*$ is a freeness relation.
\end{theo}
\begin{proof}
By Remark \ref{RemarkItIsEnoughIndependenceForOneVector}, to prove local character, finite character and transitivity it is enough to show them for the case of a 1-tuple. 
\begin{description}
   \item[Local character]
Let $v\in H$ and $E\subseteq H$. Let $w=(P_{acl(E)}(v))_e$. Then there exist a sequence of $(l_k)_{k\in\N}\subseteq \N$, a sequence of finite tuples $(a_1^k,\dots,a_{l_k}^k)_{k\in\N}\subseteq \A$ and a sequence of finite tuples $(e_1^k,\dots,e_{l_k}^k)_{k\in\N}\subseteq E$ such that if $w_k:=\sum_{j=1}^{l_k}\p(a_j^k)e_j^k$ for $k\in\N$, then $w_k\to w$ when $k\to\infty$. Let $E_0=\{e_j^k\ |\ j=1,\dots,l_k\text{ and }k\in\N\}$. Then $v\ind_{E_0}^*E$ and $|E_0|=\aleph_0$.
   \item[Finite character] We show that for $v\in H$, $E,F\subseteq H$, $v\ind^*_EF$ if and only if $v\ind^*_EF_0$ for every finite $F_0\subseteq F$. The left to right direction is clear. For right to left, suppose that $v\nind^*_EF$. Let $w=P_{acl(E\cup F)}(v)-P_{acl(E)}(v)$. Then $w\in$ acl$(E\cup F)\setminus$acl$(E)$. 
   
As in the proof of local character, there exist a sequence of pairs $(l_k,n_k)_{k\in\N}\subseteq \N^2$, a sequence of finite tuples $(a_1^k,\dots,a_{l_k+n_k}^k)_{k\in\N}\subseteq \A$ and a sequence of finite tuples $(e_1^k,\dots,e_{l_k}^k,f_1^k,\dots,f_{n_k}^k)_{k\in\N}$ such that $(e_1^k,\dots,e_{l_k}^k)\subseteq E$, $(f_1^k,\dots,f_{n_k}^k)_{k\in\N}\subseteq F$ and if $w_k:=\sum_{j=1}^{l_k}\p(a_j^k)e_i^k+\sum_{j=1}^{n_k}\p(a_{l_k+j}^k)f_j^k$ for $k\in\N$, then $w_k\to w$ when $k\to\infty$.

Since $v\nind^*_EF$, then $w=P_{acl(E\cup F)}(v)-P_{acl(E)}(v)\neq 0$. Let $\e=\|w\|>0$. Then, there is $k_\e$ such that if $k\geq k_\e$ then $\|w-w_k\|<\e$. Let $F_0:=\{f_1^1,\dots,f_{k_\e}^{n_{k_\e}}\}$, then $F_0$ is a finite subset such that $v\nind^*_EF_0$.
   \item[Transitivity of independence] Let $v\in H$ and $E\subseteq F\subseteq G\subseteq H$. If $v\ind_E^*G$ then  $P_{acl(E)}(v)=P_{acl(G)}(v)$. It is clear that $P_{acl(E)}(v)=P_{acl(F)}(v)=P_{acl(G)}(v)$ so $v\ind_E^*F$ and $v\ind_F^*G$.
Conversely, if $v\ind_E^*F$ and $v\ind_F^*G$, we have that $P_{acl(E)}(v)=P_{acl(F)}(v)$ and $P_{acl(F)}(v)=P_{acl(G)}(v)$. Then $P_{acl(E)}(v)=P_{acl(G)}(v)$ and $v\ind_E^*G$.
   \item[Symmetry]  It is clear from Remark \ref{independenceoverA}.
   \item[Invariance] Let $U$ be an automorphism of $(H,\p)$. Let $\bv=(v_1,\dots,v_n)$,$\bw=(w_1,\dots,w_n)\in H^n$ and $E\subseteq H$ be such that $\bv\ind_E^*\bw$. By Remark \ref{independenceoverA}, this means that for every $j$, $k=1,\dots,n$ $H_{P^\perp_{acl(E)}(v_j)}\perp H_{P^\perp_{acl(E)}(w_k)}$. It follows that for every $j$, $k=1,\dots,n$ $H_{P^\perp_{acl(UE)}(Uv_j)}\perp H_{P^\perp_{acl(UE)}(Uw_k)}$ and, again by Remark \ref{independenceoverA}, $Uv\ind_{acl(UE)}^*Uw$.
   \item[Existence] Let $(\tilde{H},\tilde{\p})$ be the monster model and let $E\subseteq F\subseteq \tilde{H}$ be small sets. We show, by induction on $n$, that for every $p\in S_n(E)$, there exists $q\in S_n(F)$ such that $q$ is a non-forking extension of $p$. 
		\bdc
			\item[Case $n=1$] Let $v\in \tilde{H}$ be such that $p=tp(v/E)$ and let $(H^\prime,\p^\prime,u):=(H_{(P^\perp_{acl(E)}v)},\p_{(P^\perp_{acl(E)}v)},P^\perp_{acl(E)}v)$. Then, by Fact \ref{HensonsLemma}, the model $(\hat{H},\hat{\p}):=(H,\p)\oplus (H^\prime,\p^\prime)$ is an elementary extension of $(H,\p)$. Let $v^\prime:=P_{acl(E)}v+P^\perp_{acl(E)}v_d+u\in \hat{H}$. Then, by Theorem \ref{nonforkingextension}, the type $tp(v^\prime/F)$ is a $\ind^*$-independent extension of $tp(v/E)$.
			\item[Induction step] Now, let $\bv=(v_1,\dots,v_n,v_{n+1})\in \tilde{H}^{n+1}$. By induction hypothesis, there are $v_1^\prime,\dots,v_n^\prime\in H$ such that $tp(v_1^\prime,\dots,v_n^\prime/F)$ is a $\ind^*$-independent extension of $tp(v_1,\dots,v_n/E)$. Let $U$ be an automorphism of the monster model fixing $E$ pointwise such that for every $j=1,\dots,n$, $U(v_j)=v_j^\prime$. Let $v_{n+1}^\prime\in \tilde{H}$ be such that $tp(v_{n+1}^\prime/Fv_1^\prime\cdots v_n^\prime)$ is a $\ind^*$-independent extension of $tp(U(v_{n+1})/Ev_1^\prime,\cdots v_n^\prime)$. Then, by transitivity, $tp(v_1^\prime,\dots,v_n^\prime,v_{n+1}^\prime/F)$ is a $\ind^*$-independent extension of $tp(v_1,\dots,v_n,v_{n+1}/E)$.
		\edc
   \item[Stationarity] Let $(\tilde{H},\tilde{\p})$ be the monster model and let $E\subseteq F\subseteq \tilde{H}$ be small sets. We show, by induction on $n$, that for every $p\in S_n(E)$, if $q\in S_n(F)$ is a $\ind^*$-independent extension of $p$ to $F$ then $q=p^\prime$, where $p^\prime$ is the $\ind^*$-independent extension of $p$ to $F$ built in the proof of existence.
		\bdc
			\item[Case $n=1$] Let $v\in H$ be such that $p=tp(v/E)$, and let $q\in S(F)$ and $w\in H$ be such that $w\models q$. Let $v^\prime$ be as in previous item. Then, by Theorem \ref{nonforkingextension} we have that:
		       \begin{enumerate}
        	   	\item $P_{acl(E)}v=P_{acl(F)}v^\prime=P_{acl(F)}w$
			    \item $(H_{P^\perp_{acl(E)}v},\p_{P^\perp_{acl(E)}v},P^\perp_{acl(E)}v)$ is isometrically isomorphic to both 
\[(H_{P^\perp_{acl(F)}w},\p_{P^\perp_{acl(F)}w},P^\perp_{acl(F)}w)\]
 and 
\[(H_{P^\perp_{acl(F)}v^\prime},\p_{P^\perp_{acl(F)}v^\prime},P^\perp_{acl(F)}v^\prime)\]
			    \end{enumerate}
This means that $P_{acl(F)}v^\prime=P_{acl(F)}w$ and $(H_{P^\perp_{acl(F)}w},\p_{P^\perp_{acl(F)}w},P^\perp_{acl(F)}w)$ is isometrically isomorphic to $(H_{P^\perp_{acl(F)}v^\prime},\p_{P^\perp_{acl(F)}v^\prime},P^\perp_{acl(F)}v^\prime)$ and, therefore $q=tp(v^\prime/F)=p^\prime$.
			\item[Induction step] Let $\bv=(v_1,\dots,v_n,v_{n+1})$, $\bv^\prime=(v_1^\prime,\dots,v_n,v_{n+1}^\prime)$ and $\bw=(w_1,\dots,w_{n+1})\in\tilde{H}$ be such that $\bv\models p$, $\bv^\prime\models p^\prime$ and $\bw\models q$. By transitivity, we have that $tp(v_1^\prime,\dots,v_n^\prime/F)$ and $tp(w_1,\dots,w_n/F)$ are $\ind^*$-independent extensions of $tp(v_1,\dots,v_n/E)$. By induction hypothesis, $tp(v_1^\prime,\dots,v_n^\prime/F)=tp(w_1,\dots,w_n/F)$. Let $U$ be an automorphism of the monster model fixing $E$ pointwise such that for every $j=1,\dots,n$, $U(v_j)=v_j^\prime$ and let $U^\prime$ an automorphism of the monster model fixing $F$ pointwise such that for every $j=1,\dots,n$, $U^\prime(v_j^\prime)=w_j^\prime$. Again by transitivity, $tp(U^{-1}(v_{n+1}^\prime)/Fv_1\cdots v_n)$ and $tp((U^\prime\circ U)^{-1}(w_{n+1})/Fv_1,\cdots v_n)$ are $\ind^*$-independent extensions of $tp(v_{n+1}/Ev_1,\cdots v_n)$. By the case $n=1$ $tp(U^{-1}(v_{n+1}^\prime)/Fv_1\cdots v_n)=tp((U^\prime\circ U)^{-1}(w_{n+1})/Fv_1,\cdots v_n)$ and therefore $p^\prime=tp(v_1^\prime,\dots,v_n^\prime v_{n+1}^\prime/F)=tp(w_1,\dots,w_n,w_{n+1}/F)=q$.
		\edc
\end{description}
\end{proof}

\bfc[Theorem 14.14 in \cite{BBHU}]\label{LemmaTstableIffItHasGoodIndependenceRelation}
A first order continuous logic theory $T$ is stable if and only if there is an independence relation $\ind^*$ satisfying local character, finite character of dependence, transitivity, symmetry, invariance, existence and stationarity. In that case the relation $\ind^*$ coincides with non-forking.
\efc


\bth\label{Tpisuperstable}
The theory $T_\p$ is superstable and the relation $\ind^*$ agrees with non-forking.
\eth
\bpf
By Fact \ref{LemmaTstableIffItHasGoodIndependenceRelation}, $T_\p$ is stable and the relation $\ind^*$ agrees with non-forking. To prove superstability, we have to show that for every $\bv=(v_1,\dots,v_n)\in H$, every $F\subseteq H$ and every $\e>0$, there exist a finite $F_0\subseteq F$ and $\bv^\prime=(v_1^\prime,\dots,v_n^\prime)\in H^n$ such that $\|v_j-v_j^\prime\|<\e$ and $v_j^\prime\ind_{F_0}F$ for every $j\leq n$. As in the proof of local character, for $j=1,\dots,n$ let $(\ a_1^{jk},\dots,\ a_{l_k}^{jk})_{k\in\N}$, $(\ e_1^{jk},\dots,\ e_{l_k}^{jk})_{k\in\N}$, $w_j:=(P_{acl(E)}(v_j))_e$ and $(w_j^k)_{k\in\N}$ be such that $w_j^k:=\sum_{s=1}^{l_k}\p(a_s^{jk})\ e_s^{jk}$ for $k\in\N$, and $w_j^k\to w_j$. For $j=1,\dots,n$, let $K_j\in\N$ be such that $\|w_j-w_j^{K_j}\|<\e$, let $v_j^\prime:=(P_{acl(E)}v_j)_d+w_j^{K_j}$ and let $F_0^j=\{e_s^k\ |\ k\leq K_j\text{ and }s=1,\dots,l_k\}$. If we define $F_0:=U_{j=1}^nF_0^j$, then for every $j=1,\dots,n$ we have that $v_j^\prime\ind_{F_0}^*F$, $|F_0|<\aleph_0$ and $\|v_j-v_j^\prime\|<\e$.
\epf


Recall that a canonical base for a type $p$ is a minimal set over which $p$ does not fork. In general, this smallest tuple is an imaginary, but in Hilbert spaces  it corresponds to a tuple of real elements. Next theorem gives an explicit description of canonical bases for types in the structure, again we get a tuple of real elements.
\begin{theo}\label{canonicalbase}
Let $\bv=(v_1,\dots,v_n)\in H^n$ and $E\subseteq H$. Then $Cb(tp(\bv/E)):=\{(P_Ev_1,\dots,P_Ev_n)\}$ is a canonical base for the type $tp(\bv/E)$.
\end{theo}
\begin{proof}
First of all, we consider the case of a 1-tuple. By Theorem \ref{nonforkingextension} $tp(v/E)$ does not fork over $Cb(tp(v/E))$. Let $(v_k)_{k<\omega}$ a Morley sequence for $tp(v/E)$. We have to show that $P_Ev\in dcl((v_k)_{k<\omega})$. By Theorem \ref{nonforkingextension}, for every $k<\omega$ there is a vector $w_k$ such that $v_k=P_Ev+w_k$ and $w_k\perp acl(\{P_Ev\}\cup\{w_j\ |\ j< k\})$. This means that for every $k<\omega$, $w_k\in H_e$ and for all $j$, $k<\omega$, $H_{w_j}\perp H_{w_k}$. For $k<\omega$, let $v^\prime_k:=\frac{v_1+\cdots+v_k}{n}=P_Ev+\frac{w_1+\cdots+w_k}{n}$. Then for every $k<\omega$, $v^\prime_k\in dcl((v_k)_{k<\omega})$. Since $v^\prime_k\to P_ev$ when $k\to\infty$, we have that $P_Ev\in dcl((v_k)_{k<\omega})$.

For the case of a general $n$-tuple, by Remark \ref{RemarkItIsEnoughIndependenceForOneVector}, it is enough to repeat the previous argument in every component of $\bv$.
\end{proof}

\bco
The theory $T_\p$ has weak elimination of imaginaries.
\eco
\bpf
Clear by previous theorem.
\epf

\section{orthogonality and domination}\label{OrthogonalityAndDominationOfTypes}
In this section, we characterize domination and orthogonality of types in terms of similar relationships between positive linear functionals on $\A$. These are the statements Theorem \ref{orthogonalityoverA} and Theorem \ref{dominationoverA}. For a complete description of the relation of domination see \cite{Bue}, Definition 5.6.4.

\begin{defin}
Let $\A^\prime$ be the dual space of $\A$. An element $\f\in\A^\prime$ is called \textit{positive} if $\f(a)\geq 0$ whenever $a\in\A$ is positive, i.e. there is $b\in\A$ such that $a=b^*b$. The set of positive functionals is denoted by $\A^\prime_+$.
\end{defin}

\begin{lemma}\label{vectorsdefinepositivelinearfunctionals}
Let $\A$ be a $C^*$-algebra of operators on a Hilbert space $H$, and let $v\in H$. Then the function $\f_v$ on $\A$ such that for every $S\in\A$, $\f_v(S)=\langle Sv\,|\,v\rangle$ is a positive linear functional.
\end{lemma}

Recall Theorem 4.4 in \cite{Ar}
\begin{theo}\label{typeoverempty4}
Let $v,w\in H$. Then $tp(v/\emptyset)=tp(w/\emptyset)$ if and only if $\f_v=\f_w$, where $\f_v$ denotes the positive linear functional on $\A$ defined by $v$ as in Lemma \ref{vectorsdefinepositivelinearfunctionals}.
\end{theo}

\begin{defin}\label{relationsbetweenpositivelinearfunctionals}
Let $\f$ and $\psi$ be positive linear functionals on $\A$.
\ben
\item They are called \textit{orthogonal} ($\f\perp\psi$) if $\|\f-\psi\|=\|\f\|+\|\psi\|$.
\item Also, $\f$ is called \textit{dominated} by $\psi$ ($\f\leq\psi$) if there exist $\gamma>0$ such that the functional $\gamma\psi-\f$ is positive.
\een
\end{defin}

\begin{theo}\label{HvinHw}
Let $v,w\in H$. Then $(H_v,\p_v,v)$ is isometrically isomorphic to a subrepresentation of $(H_w,\p_w,w)$ if and only if $\f_v\leq\f_w$.
\end{theo}
\bpf
Suppose $(H_v,\p_v,v)$ is isometrically isomorphic to a subrepresentation of $(H_w,\p_w,w)$. Then there exists a vector $v^\prime\in H_w$ such that $(H_v,\p_v,v)\simeq (H_{v^\prime},\p_{v^\prime},v^\prime)$. By Radon Nikodim Theorem for rings of operators (see \cite{Hen}), there exists a bounded positive operator $P:(H_w,\p_w,w)\to (H_{v^\prime},\p_{v^\prime},v^\prime)$ such that $Pw=v^\prime$ and $P$ commutes with every element of $\p_v(\A)$. Let $\g=\|P\|^2$. Then, for every positive element $a\in\A$, $\f_{v}(a)=\f_{v^\prime}(a)=\langle \p(a)v^\prime\ |\ v^\prime\rangle=\langle \p(a)Pw\ |\ Pw\rangle=\langle P^*\p(a)Pw\ |\ w\rangle=\langle \p(a)\|P\|^2w\ |\ w\rangle\leq \g\langle \p(a)w\ |\ w\rangle=\g\f_{w}(a)$ which means that $\g\f_w-\f_{v}$ is positive and $\f_v\leq\f_w$.

The converse is Corollary 3.3.8 in \cite{Pe}.
\epf

Recall the important GNS theorem:

\bth[Theorem 3.3.3. and Remark 3.4.1. in \cite{Pe}]\label{GelfandNaimarkSegal}
Let $\f$ be a positive functional on $\A$. Then there exists a cyclic representation $(H_\f,\p_\f,v_\f)$ such that for all $a\in\A$, $\f(a)=\langle \p_\f(a)v_\f|v_\f\rangle$. This representation is called the \textit{Gelfand-Naimark-Segal construction}.
\eth

\begin{lemma}\label{phivperpphiwnotismoetricallyisomorphic}
Let $v,w\in H$. If $\f_v\perp\f_w$, then $(H_v,\p_v,v)$ is not isometrically isomorphic to any subrepresentation of $(H_w,\p_w,w)$.
\end{lemma}
\bpf
Suppose $\f_v\perp\f_w$, and $(H_v,\p_v,v)$ is isometrically isomorphic to subrepresentation of $(H_w,\p_w,w)$. By Theorem \ref{HvinHw} $\f_v\leq\f_w$; let $\g>0$ be a real number such that $\g\f_w-\f_v$ is a bounded positive functional and let $u\in H$ be such that $\f_u=\g\f_w-\f_v$, which is possible by GNS Theorem. Then $\f_v=\g\f_w-\f_u$, and $\|\f_w-\f_v\|=\|\f_w-\g\f_w+\f_u\|=\|(1-\g)\f_w+\f_u\|=|1-\g|\|\f_w\|+\|\f_u\|\neq \|\f_w\|+\|\f_v\|$, but this contradicts $\f_v\perp\f_w$. 
\epf

Here, a few facts that will be needed to prove Theorem \ref{HvperpHw}:

\brm
Recall that two representations are said to be \textit{disjoint} if they do not have any common subrepresentation up to isometric isomorphism.
\erm

\bfc[Proposition 3 in \cite{Co}, Chapter 5, Section 2]\label{FactDisjointSubrepresentationsHaveDisjointProjections}
Two subrepresentations $(H_1,\p_1)$, $(H_2,\p_2)$ of $(H,\p)$ are disjoint if and only if there is a projection $P$ in $\p(\A)^\prime\cap \p(\A)^{\prime\prime}$ such that if $P_1$ and $P_2$ are the projections on $H_1$ and $H_2$ respectively, we have that $PP_1=P_1$ and $(I-P)P_2=P_2$.
\efc

\bfc[Corollary 2.2.5 in \cite{Pe}]\label{FactMMisDoubleCommutantOfPiOfA}
Let $\A$ be a $C^*$algebra and $\p:\A\to B(H)$ a nondegenerate representation of $\A$. Let $\MM$ be the strong closure of $\pi(\A)$. Then $\MM$ is weakly closed and $\MM=\A^{\prime\prime}$.
\efc

\begin{theo}[Kaplanski densitiy theorem. Theorem $2.3.3.$ in \cite{Pe}]\label{TheoremKaplansky}
Let $\A$ be a $C^*$-subalgebra of $B(H)$ with strong closure $\MM$. Then the unit ball $Ball_1(\A)$ of $\A$ is strongly dense in the unit ball $Ball_1(\MM)$ of $\MM$. Furthermore, the set of selfadjoint elements in $Ball_1(\A)$ is strongly dense in the set of selfadjoint elements of $Ball_1(\MM)$.
\end{theo}

\bfc[Lemma 3.2.3 in \cite{Pe}]\label{phiperppsi}
Let $\f$ and $\psi$ be two positive linear functionals on $\A$. Then, $\f\perp\psi$ if and only if for all $\e>0$ there exists a positive element $a\in\A$ with norm less than or equal to $1$, such that $\f(e-a)<\e$ and $\psi(a)<\e$.
\efc

\blm\label{phi1lessthanphi2perppsi1lessthatpsi2}
Let $\f_1$, $\f_2$, $\psi_1$ and $\psi_2$ be positive linear functionals on $\A$ such that $\f_1\leq\f_2$ and $\psi_1\leq\psi_2$. If $\f_2\perp\psi_2$, then $\f_1\perp\psi_1$.
\elm
\bpf
Let $\g_1>0$ and $\g_2>0$ be such that $\g_1\f_2-\f_1$ and $\g_2\psi_2-\psi_1$ are positive. By Fact \ref{phiperppsi}, for $\e>0$ there exists a positive $a\in\A$ with norm less than or equal to $1$ such that $\f_2(e-a)<\frac{\e}{\g_1+\g_2}$ and $\psi_2(a)<\frac{\e}{\g_1+\g_2}$.
Then $\f_1(e-a)\leq\g_1\f_2(e-a)<\frac{\g_1\e}{\g_1+\g_2}<\e$ and $\psi_1(a)\leq\g_2\psi_2(a)<\frac{\g_2\e}{\g_1+\g_2}<\e$. 
\epf

\begin{theo}\label{HvperpHw}
Let $v,w\in H$. $\f_v\perp\f_w$ if and only if no subrepresentation of $(H_v,\p_v,v)$ is isometrically isomorphic to a subrepresentation of $(H_w,\p_w,w)$.
\end{theo}
\bpf
Suppose $\f_v\perp\f_w$. By Lemma \ref{phi1lessthanphi2perppsi1lessthatpsi2}, if $(H_{v^\prime},\p_{v^\prime},v^\prime)$ is a subrepresentation of $(H_v,\p_v,v)$ and $(H_{w^\prime},\p_{w^\prime},w^\prime)$  is a subrepresentation of $(H_w,\p_w,w)$, then $\f_{v^\prime}\perp\f_{w^\prime}$, By Lemma \ref{phivperpphiwnotismoetricallyisomorphic}, $(H_{v^\prime},\p_{v^\prime},v^\prime)$ is not isometrically isomorphic to $(H_{w^\prime},\p_{w^\prime},w^\prime)$, and the conclusion follows.

Conversely, suppose no subrepresentation of $(H_v,\p_v,v)$ is isometrically isomorphic to a subrepresentation of $(H_w,\p_w,w)$. Then the representations $(H_v,\p_v)$ and $(H_w,\p_w)$ are disjoint. By Fact \ref{FactDisjointSubrepresentationsHaveDisjointProjections}, there is a projection $P\in\p(\A)^\prime\cap\p(\A)^{\prime\prime}$ such that $PP_v=P_v$ and $(I-P)P_w=P_w$. Then, $\f_v(I-P)=\langle (I-P)v\ |\ v\rangle=\langle (v-PP_vv)\ |\ v\rangle=\langle (v-v)\ |\ v\rangle=0$. On the other hand, $\f_w(P)=\langle Pw\ |\ w\rangle=\langle w-(w-Pw)\ |\ w\rangle=\langle w-(I-P)w\ |\ w\rangle=\langle w-(I-P)P_ww\ |\ w\rangle=\langle w-P_ww\ |\ w\rangle=\langle w-w\ |\ w\rangle=0$. By Fact \ref{FactMMisDoubleCommutantOfPiOfA} and Theorem \ref{TheoremKaplansky}, the projection $P$ is strongly approximable by positive elements in $\p(\A)$ and therefore, for $\e>0$ there exists a positive element $a\in\A$ with norm less than or equal to $1$, such that $\f_v(e-a)<\e$ and $\f_w(a)<\e$. By Fact \ref{phiperppsi}, $\f_v\perp\f_w$.
\epf

\begin{lemma}\label{orthogonalityoverempty}
Let $p,q\in S_1(\emptyset)$, let $v$, $w\in H$ be such that $v\models p$ and $w\models q$.
Then, $p\perp^a q$ if and only if $\f_{v_e}\perp \f_{w_e}$.
\end{lemma}
\begin{proof}
Suppose $p\perp^a q$. By Remark \ref{independenceoverempty}, this implies that $H_{v_e}\perp H_{w_e}$ for all $v\models p$ and $w\models q$. Let $v\models p$ and $w\models q$. Then no subrepresentation of $(H_{v_e},\p_{v_e},v_e)$ is isometrically isomorphic to any subrepresentation of $(H_{w_e},\p_{w_e},w_e)$. By Lemma \ref{HvperpHw}, this implies that $\f_{v_e}\perp \f_{w_e}$.

Conversely, if $p\not\perp^a q$ there are $v$, $w\in H$ such that $v\models p$, $w\models q$ and $H_{v_e}\not\perp H_{w_e}$. This implies that there exist elements $a_1$, $a_2\in\A$ such that $\p(a_1)v_e\not\perp\p(a_2)w_e$. This means that $0\neq\langle\p(a_1)v_e\ |\ \p(a_2)w_e \rangle=\langle v_e\ |\ \p(a_1^*a_2)w_e \rangle$. So, we can assume that there exists an element $a\in\A$ such that $v_e\not\perp\p(a)w_e$. Since $v_e=P_{w_e}v_e+P_{w_e}^\perp v_e$ and $P_{w_e}v_e\neq 0$, we can prove that $\f_{P_{w_e}v_e}\leq \f_{v_e}$ by using a procedure similar to the one used in the proof of Theorem \ref{HvinHw} and, since $P_{w_e}v_e\in H_{w_e}$, we get $\f_{P_{w_e}v_e}\leq \f_{w_e}$. By Lemma \ref{phi1lessthanphi2perppsi1lessthatpsi2}, this implies that $\f_{v_e}\not\perp\f_{w_e}$.
\end{proof}

\begin{theo}\label{orthogonalityoverA}
Let $E\subseteq H$. Let $p,q\in
S_1(E)$, let $v$, $w\in H$ be such that $v\models p$ and $w\models q$. Then, $p\perp_E^a q$ if and only if $\f_{P^\perp_E(v_e)}\perp \f_{P^\perp_E(w_e)}$
\end{theo}
\begin{proof}
Clear by Lemma \ref{orthogonalityoverempty}.
\end{proof}

\begin{theo}
Let $E\subseteq H$. Let $p,q\in
S_1(E)$. Then, $p\perp^a q$ if and only if $p\perp q$.
\end{theo}
\bpf
Assume $p\perp^a q$, $E\subseteq F\subseteq H$ are small subsets of the monster model and $p^\prime$, $q^\prime\in S_1(F)$ are non-forking extensions of $p$ and $q$ respectively. Let $v$, $w\in H$ be such that $v\models p^\prime$ and $w\models q^\prime$, then $\f_{P^\perp_F(v_e)}=\f_{P^\perp_Ev_e}\perp\f_{P^\perp_Ew_e}=\f_{P^\perp_F(w_e)}$. By Lemma \ref{orthogonalityoverempty}, this implies that $p^\prime\perp^a q^\prime$. Therefore $p\perp q$.

The converse is trivial.
\epf

\begin{lemma}\label{dominationoverempty}
Let $p,q\in S_1(\emptyset)$ and let $v$, $w\in H$ be such that $v\models p$ and $w\models q$.
Then, $p\vartriangleright_\emptyset q$ if and only if $\f_{w_e}\leq\f_{v_e}$.
\end{lemma}
\begin{proof}
Suppose $p\vartriangleright_\emptyset q$. Suppose that $v^\prime$ and $w^\prime$ are such that $v^\prime\models p$, $w^ \prime\models q$ and if $v^\prime\ind^*_\emptyset E$ then $w^\prime\ind_\emptyset^*E$ for every $E$. Then for every $E\subseteq H$
\[P_Ev^\prime_e=0\Rightarrow P_Ew^\prime_e=0\]
This implies that $w^\prime_e\in H_{v^\prime_e}$, and $H_{w^\prime_e}\subseteq H_{v^\prime_e}$. By Theorem \ref{HvinHw}, $\f_{w_e}=\f_{w^\prime_e}\leq\f_{v^\prime_e}=\f_{v_e}$.

For the converse, suppose $\f_{w_e}\leq\f_{v_e}$. Then, by Theorem \ref{HvinHw} $H_{w_e}$ is isometrically isomorphic to a subrepresentation of $H_{v_e}$, which implies that there is $w^\prime\in H_v$ such that $w^\prime\models tp(w/\emptyset)$ and for every $E\subseteq H$
\[P_Ev_e=0\Rightarrow P_Ew^\prime_e=0\]
This means than $tp(w/\emptyset)\vartriangleleft_\emptyset tp(v/\emptyset)$.
\end{proof}

\begin{theo}\label{dominationoverA}
Let $E$, $F$ and $G$ be small subsets of $\tilde{H}$ such that $E$, $F\subseteq G$ and $p\in S_1(E)$ and $q\in S_1(F)$ be two stationary types. Then $p\vartriangleright_G q$ if and only if there exist $v$, $w\in\tilde{H}$ such that $tp(v/G)$ is a non-forking extension of $p$, $tp(w/G)$ is a non-forking extension of $q$ and $\f_{P^\perp_{acl(G)}w_e}\leq\f_{P^\perp_{acl(G)}v_e}$.
\end{theo}
\begin{proof}
Clear by Lemma \ref{dominationoverempty}.
\end{proof}

\end{document}